\documentclass{amsart}
\usepackage{amsmath}
\usepackage{amsfonts}
\usepackage{graphics}
\usepackage{epsfig}
\usepackage{amssymb}
\usepackage{amscd}
\usepackage[all]{xy}
\usepackage{latexsym}
\usepackage{graphicx}
\usepackage{multirow}
\usepackage{geometry}
\usepackage{color}
\usepackage[usenames,dvipsnames]{pstricks}
\usepackage{epsfig}
\usepackage{pst-grad} 
\usepackage{pst-plot}

\newtheorem{thm}{Theorem}[section]
\newtheorem{cor}[thm]{Corollary}
\newtheorem{lem}[thm]{Lemma}
\newtheorem{prop}[thm]{Proposition}

\newtheorem{quest}[thm]{Question}

\theoremstyle{definition}
\newtheorem{Def}[thm]{Definition}
\newtheorem{rem}[thm]{Remark}
\newtheorem*{ack}{Acknowledgement}
\newtheorem{ex}[thm]{Example}

\numberwithin{equation}{section}
\numberwithin{figure}{section}

\def\Hom{{\text{\rm{Hom}}}}

\def\rchi{{\hbox{\raise1.5pt\hbox{$\chi$}}}}

\def\tensor{\otimes}

\def\a{\alpha}
\def\b{\beta}
\def\lam{\lambda}

\def\gam{\gamma}
\def\Gam{\Gamma}

\def\Vect{{\rm{\bf{Vect}}}}

\newcommand{\bea}{\begin{eqnarray}}
\newcommand{\eea}{\end{eqnarray}}
\newcommand{\be}{\begin{equation}}
\newcommand{\ee}{\end{equation}}

\newcommand{\Mbar}{{\overline{\mathcal{M}}}}

\newcommand{\bP}{{\mathbb{P}}}
\newcommand{\bC}{{\mathbb{C}}}

\newcommand{\bL}{{\mathbb{L}}}

\newcommand{\bQ}{{\mathbb{Q}}}
\newcommand{\bR}{{\mathbb{R}}}

\newcommand{\cA}{{\mathcal{A}}}

\newcommand{\cM}{{\mathcal{M}}}

\newcommand{\cC}{{\mathcal{C}}}

\newcommand{\cF}{{\mathcal{F}}}
\newcommand{\cG}{{\mathcal{G}}}

\newcommand{\cT}{{\mathcal{T}}}

\newcommand{\la}{{\langle}}
\newcommand{\ra}{{\rangle}}

\newcommand{\lrar}{\longrightarrow}

\begin{document}
\large
\setcounter{section}{0}

\allowdisplaybreaks

\title[Edge contraction and 2D TQFT]
{Edge contraction on dual ribbon graphs and 2D TQFT}

\author[O.\ Dumitrescu]{Olivia Dumitrescu}
\address{O.~Dumitrescu: Max-Planck-Institut
f\"ur Mathematik, Bonn, Germany}
\address{Current Address: 
Department of Mathematics\\
Central Michigan University\\
Mount Pleasant, MI 48859}
\email{dumit1om@cmich.edu}
\address{Simion Stoilow Institute of Mathematics\\
Romanian Academy\\
21 Calea Grivitei Street\\
010702 Bucharest, Romania}

\author[M.\ Mulase]{Motohico Mulase}
\address{M.~Mulase:
Department of Mathematics\\
University of California\\
Davis, CA 95616--8633}
\email{mulase@math.ucdavis.edu}
\address{Kavli Institute for Physics and Mathematics of the 
Universe\\
The University of Tokyo\\
Kashiwa, Japan}

\begin{abstract}
We present a new set of 
axioms for  2D TQFT
formulated on the category of   
cell graphs with edge-contraction operations
as morphisms. We construct a functor
from this category to the endofunctor
category consisting of Frobenius algebras. 
Edge-contraction operations
 correspond to natural transformations
of endofunctors, which are compatible with
the Frobenius algebra structure.
Given a Frobenius algebra $A$,
every cell graph 
determines an element of 
the symmetric tensor algebra defined over 
the dual space $A^*$.
We show that the edge-contraction axioms make
this assignment depending  only on the topological 
type of the cell graph,
but not on the graph itself. Thus the functor 
generates the TQFT corresponding to $A$.
\end{abstract}

\subjclass[2010]{Primary: 14N35, 81T45, 14N10;
Secondary: 53D37, 05A15}

\keywords{Topological quantum field 
theory;  Frobenius 
algebras; ribbon graphs; cell graphs}

\maketitle

\tableofcontents

\section{Introduction}
\label{sect:intro}

The purpose of the present paper is  to give 
a new set of axioms for  two-dimensional
topological quantum field theory
(2D TQFT)  formulated
in terms of  dual  ribbon graphs.
The key relations between ribbon graphs are 
\textbf{edge-contraction operations}, which
correspond to the degenerations in
 the moduli space $\Mbar_{g,n}$ of
stable  curves of genus $g$ with $n$
 labeled points that create a rational 
 component with $3$ special points. The 
 structure of Frobenius
 algebra is naturally encoded in the category
  of dual ribbon
  graphs, where edge-contraction
  operations form
  morphisms and 
represent   multiplication 
 and comultiplication operations.

As Grothendieck impressively presents in 
\cite{Gro}, it is a beautiful and simple yet 
 very surprising idea that
 a graph drawn on a compact
topological surface gives an algebraic structure
to the surface. When a positive real 
number is assigned to each edge as its  length,
a unique complex structure of the surface is
determined. This association leads to 
a combinatorial model for the moduli
space $\cM_{g,n}$ of smooth  algebraic
curves of genus $g$ with $n$ marked points
\cite{Harer,  MP1998, Mumford,STT,Strebel}.   
By identifying these graphs
as Feynman diagrams  of  \cite{tH} appearing in the
asymptotic expansion 
of  a particular matrix integral,
and by giving a graph description of 
tautological cotangent classes on $\Mbar_{g,n}$,
Kontsevich \cite{K1992} shows that Witten's generating 
function \cite{W1991} of  intersection numbers 
of these classes satisfies 
the KdV equations. 
Kontsevich's argument is
based on his discovery  that an weighted sum of 
these intersection 
numbers is proportional to the Euclidean volume
of the combinatorial model of $\cM_{g,n}$.

The Euclidean volume 
of $\cM_{g,n}$ depends on the choice of 
the perimeter length of each face of the graph
drawn on a surface. Kontsevich used the 
\emph{Laplace
transform} of the volume as a function
of the perimeter length to obtain a set of relations
among 
intersection numbers of different
values of $(g,n)$. These relations are equivalent
to the conjectured KdV equations. 

Recall that if each edge has an integer length, 
then the resulting Riemann surface
by the Strebel correspondence
\cite{Strebel}  is an algebraic 
curve defined over 
$\overline{\bQ}$ \cite{Belyi, MP1998}. Thus a systematic
counting of curves defined over $\overline{\bQ}$
gives an  approximation 
of the Euclidean volume of Kontsevich
by lattice point counting. Since these
lattice points naturally correspond to the graphs
themselves, the intersection numbers in 
question can be obtained by graph enumeration,
after taking the limit as the mesh length approaches 
to $0$. 
Now we note that
edge-contraction operations give an
 effective tool for graph enumeration problems. 
 Then one can ask: 
 \emph{what information do the
edge-contraction operations tell us about
the intersection numbers?}

We found in \cite{OM4,DMSS,MP2012}
that the Laplace transform of the counting
formula obtained by the edge-contraction
operations on graphs is exactly 
the Virasoro constraint conditions of
\cite{DVV}  for the
intersection numbers. Indeed it gives
the most fundamental example of 
 \emph{topological recursion} of
\cite{EO2007}.

Euclidean volume is naturally
 approximated by 
lattice point counting. It can be also 
approximated as a limit of hyperbolic volume.
The latter idea applied to moduli spaces
of hyperbolic surfaces gives the 
same Virasoro constraint conditions, as beautifully
described in the work of Mirzakhani \cite{Mir1,Mir2}.
Mirzakhani's technique of symplectic and
hyperbolic geometry can be naturally 
extended to \emph{character varieties}
of surface groups. Yet there are no
Virasoro constraints  for this type of moduli 
spaces. 
We ask: \emph{what do edge-contraction
operations give us for the character varieties?}

This is our motivation of the current paper. 
Instead of discussing the application of 
our result to character varieties, which will
be carried out elsewhere, we focus
in this paper our discovery of the relation between 
edge-contraction
operations and
2D TQFT.


A TQFT
 of dimension $d$ is a 
 symmetric monoidal functor $Z$
 from the monoidal category of $(d-1)$-dimensional
 compact oriented topological manifolds, 
 with $d$-dimensional oriented cobordism forming
 morphisms among $(d-1)$-dimensional boundary
 manifolds,
 to the monoidal category of
 finite-dimensional vector spaces defined over a 
 fixed field $K$ \cite{Atiyah, Segal}. 
 Since there is only one compact
 manifold in dimension $1$, a 2D
 TQFT is associated with a unique vector space
 $A =Z(S^1)$, and the Atiyah-Segal axioms of TQFT
 makes $A$ a commutative Frobenius algebra. 
 It has been established
 that  2D TQFTs are classified by finite-dimensional
 Frobenius algebras \cite{Abrams, Dijkgraaf}.
 We ask the following question, in the reverse 
 direction:
 
 \begin{quest}
Suppose we are given a finite-dimensional commutative 
Frobenius algebra. What is the 
 combinatorial realization of the
algebra structure that leads to the corresponding
2D TQFT?
\end{quest}

The answer we propose in this paper is 
the \emph{category of  dual ribbon graphs}, with 
edge-contraction operations as morphisms. 
This category does not carry the information of 
a specific Frobenius algebra. 
In our forthcoming paper, 
we will show that our category 
 generates all
Frobenius objects among any
given monoidal category.

For a given
Frobenius algebra $A$
and  a ribbon graph $\gam_{g,n}$
with $n$ vertices drawn on
a topological surface of genus $g$,
we assign  a 
multilinear map
$$
\gam_{g,n}:A^{\tensor n}\lrar K.
$$
The \emph{edge-contraction axioms} 
of Section~\ref{sect:ECA}
determine the behavior of this map under the change
of ribbon graphs via edge contractions. 
Theorem~\ref{thm:independence}, our main 
result of this paper,
exhibits a surprising statement that the 
map $\gam_{g,n}$ depends only on $g$ and $n$,
and is
 independent of the choice of 
 the graph $\gam_{g,n}$. We
then evaluate $\gam_{g,n}$
for each 
$v_1\tensor\cdots\tensor v_n\in A^{\tensor n}$
and prove 
that this map indeed defines the TQFT corresponding
to $A$.
 
 A \emph{ribbon graph}
 (also called as a dessin d'enfant, fatgraph,
 embedded graph, or a map) is a graph with an assignment of
 a cyclic order of  half-edges incident at each 
 vertex. The cyclic order induces the ribbon 
 structure to the graph, and it becomes 
 the $1$-skeleton of the cell-decomposition of a
compact oriented topological 
surface of genus, say $g$, by 
attaching oriented open discs to the graph. Let 
$n$ be the number of the discs attached. We call
this ribbon graph of \emph{type} $(g,n)$. 

\begin{figure}[htb]
\includegraphics[height=0.7in]{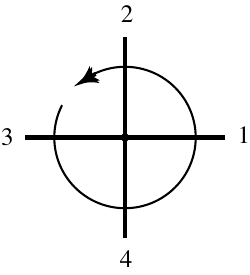}
\hskip0.5in
\includegraphics[height=0.7in]{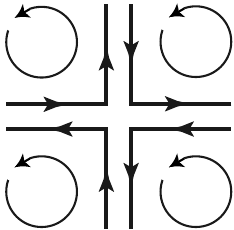}
\\
\medskip
\includegraphics[height=0.6in]{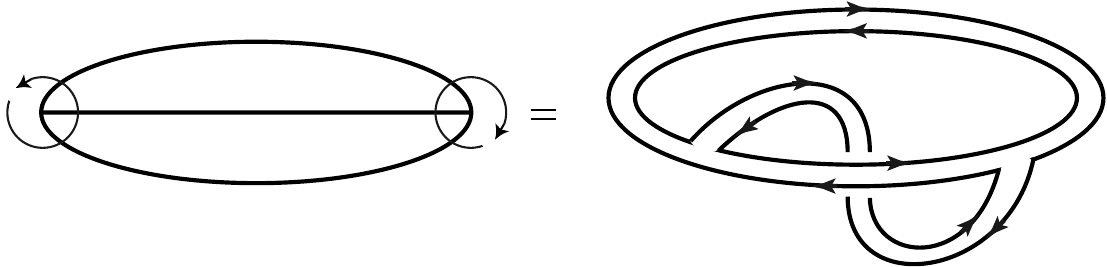}
\\
\includegraphics[height=0.6in]{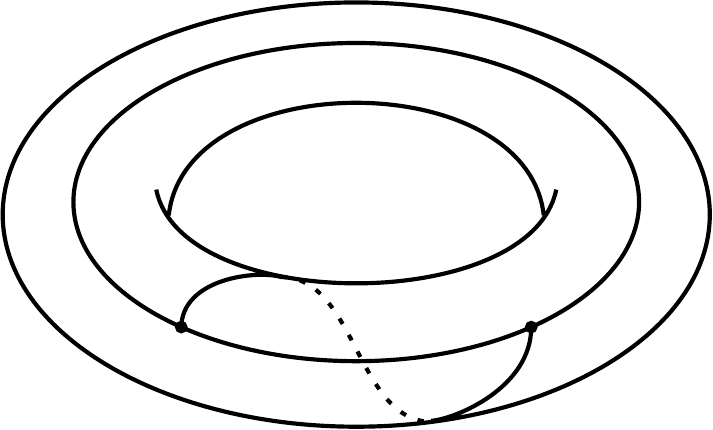}
\caption{Top Row: A cyclic order of half-edges
at a vertex induces a local ribbon structure
to a graph. Second Row:
Globally, a ribbon graph is
the $1$-skeleton of a
cell-decomposition of
a  compact oriented surface.
Third Row: A ribbon graph is thus a graph drawn
on a compact oriented surface.
}
\end{figure}

 An assignment of a positive real number to each
 edge of a ribbon graph determines a concrete
 holomorphic coordinate system of the topological 
 surface of genus $g$ with $n$ labeled marked points
 \cite{MP1998}, thus making it a Riemann surface.
 This construction gives the identification of
 the space of ribbon graphs of type 
 $(g,n)$ with positive  
 edge lengths assigned, and 
 the space $\cM_{g,n}\times \bR_+^n$,
 as an orbifold. The operation of edge-contraction 
 of an edge connecting two distinct vertices then 
 defines the boundary operator, which introduces
 the structure of 
  orbi-cell complex
 on $\cM_{g,n}\times \bR_+^n$. Each ribbon graph determines
the stratum of this cell complex, whose dimension
is the number of edges of the graph.

Since the
ribbon graphs we need for the consideration
of TQFT have \emph{labeled vertices} 
but no labels for faces, we use the terminology
\textbf{cell graph 
of type} $(g,n)$
for a ribbon
graph of genus $g$ with $n$ labeled vertices. 
A cell graph of type $(g,n)$ is the dual of
a ribbon graph of the same type $(g,n)$. 
The set of all cell graphs of type $(g,n)$ is
denoted by $\Gam_{g,n}$. 

Ribbon graphs naturally form orbi-cell complex.
Their dual cell graphs naturally form a category
$\cC\cG$, as we shall define in 
Section~\ref{sect:category}.
We then consider functors
$$
\omega:\cC\cG \lrar
\cF un(\cC/K,\cC/K),
$$
where $(\cC,\tensor,K)$ is a monoidal 
category with the unit object $K$, 
and $\cF un(\cC/K,\cC/K)$ is the endofunctor
category over the category of $K$-objects
of $\cC$. Each cell graph
corresponds to an endofunctor, 
and  edge-contraction operations among them
correspond to natural transformations. 
 Our consideration 
 can be generalized to the cohomological field
 theory of Kontsevich-Manin \cite{KM}. 
 After this generalization, we can construct 
 a functor that gives a classification of 2D TQFT.
 Since we need more preparation,
 these topics will be 
 discussed in our forthcoming paper.

Edge-contraction operations 
 also provide an effective method
for  graph enumeration problems. 
It has been noted in 
\cite{DMSS} that the Laplace transform
of  edge-contraction operations on many
counting problems corresponds
to the topological recursion of
\cite{EO2007}. In
a separate paper
\cite{OM7}, we  give the 
construction of the mirror B-models corresponding
to  the simple
and orbifold Hurwitz numbers, by using only
the edge-contraction operations. 
In general,
enumerative geometry problems, 
such as computation of Gromov-Witten type
invariants, are  solved by 
studying a corresponding problem
on the  \emph{mirror dual} side. 
The effectiveness of the mirror problem  relies
on the technique of complex analysis. 
The  question is:
How do we find the mirror of a given
enumerative problem?
In \cite{OM7},
we give an answer to this question
for a class of graph enumeration problems
that are equivalent to counting of 
orbifold Hurwitz numbers. 
The key is again the same edge-contraction operations. 
The base case, or the case for the ``moduli space''
$\Mbar_{0,1}$, of the edge contraction in the
counting problem identifies the mirror dual
object, and a universal mechanism
of complex analysis, known as the
\textbf{topological recursion} of \cite{EO2007},
solves the B-model side of the counting
problem. The solution is a collection of
generating functions of the original problem
for all genera.

The edge-contraction operation 
causes the degeneration of $\bP^1$ with
one marked point $p$ into two $\bP^1$'s
with one marked point on each, connected
by a $\bP^1$ with $3$ special points, 
two of which are nodal points and
the third one representing the original 
marked point $p$. 
In terms of graph enumeration, the $\bP^1$ with
$3$ special points does not play any role. So we
break the original vertex into two vertices, and 
 separate the graph into two disjoint pieces
(Figure~\ref{fig:01}).

\begin{figure}[htb]
\includegraphics[width=2in]{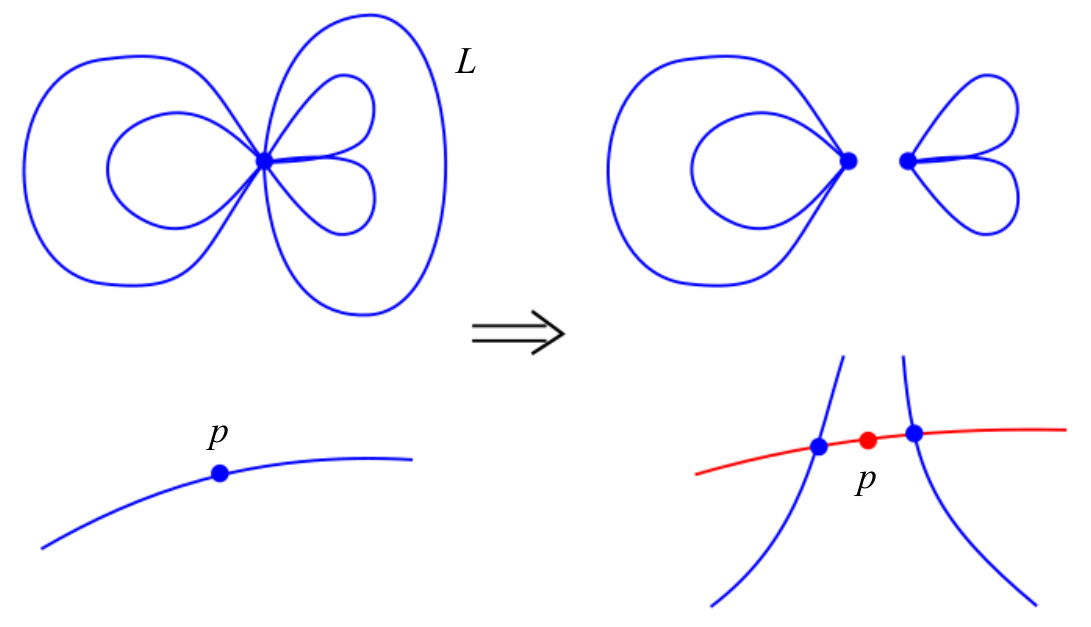}
\caption{The edge-contraction operation 
on a loop is a degeneration process. The graph
on the left is a connected cell graph of type 
$(0,1)$. The edge-contraction on the loop 
$L$ changes it to the one on the right.
Here,
a $\bP^1$ with one marked point $p$
degenerates into two $\bP^1$'s with one
marked point on each, connected by
a $\bP^1$ with $3$ special points.
}
\label{fig:01}
\end{figure}

Once we have our formulation of
2D TQFT and topological recursion in terms
of edge-contraction operations, 
we can consider a TQFT-valued topological 
recursion. An immediate example 
is the Gromov-Witten theory of 
the classifying space $BG$ of a finite group $G$.
In our forthcoming paper, 
we  will show that 
a straightforward generalization of 
the topological recursion 
for differential forms with values in 
tensor products of a Frobenius algebra
 automatically splits into the product
of the usual scalar-valued 
solution to the topological recursion
and a 2D TQFT. Therefore, topological recursion
implies TQFT.
Here, we remark the 
similarity between the topological recursion
 and the comultiplication
operation in a Frobenius algebra. 
Indeed, the topological 
recursion itself can be
regarded as a comultiplication 
formula for an infinite-dimensional
analogue of the Frobenius algebra 
(Vertex algebras, 
or conformal field theory).

The authors have noticed that the topological recursion
appears as the Laplace transform of edge-contraction 
operations in \cite{DMSS}. The geometric
nature of the topological recursion was further
investigated in \cite{OM1,OM2, OM5}, where it
was placed in the context of Hitchin spectral curves
for the first time, and the relation to quantum curves 
was discovered. The present paper is the authors'
first step toward identifying the topological recursion 
in an algebraic and categorical setting.
We note that Hitchin moduli spaces are
diffeomorphic to character varieties of a surface
group. The TQFT point of view of our current
paper in the context of these character varieties,
in particular, their Hodge structures, will 
be discussed elsewhere.

The paper is organized as follows. 
We start with a quick review of Frobenius algebras,
for the purpose of setting notations, in
Section~\ref{sect:Frobenius}. 
We then recall two-dimensional TQFT in
Section~\ref{sect:tqft}.
In  Sections~\ref{sect:ECA},
 we give our formulation
of 2D TQFT in terms of the 
edge-contraction axioms of cell graphs. 
A categorical formulation of our 
axioms is given  in 
Section~\ref{sect:category}.

\section{Frobenius algebras}
\label{sect:Frobenius}

In this paper,
we are concerned with 
finite-dimensional,
unital,   commutative
Frobenius algebras defined over a field 
$K$. 
In this section we review the necessary 
account of 
Frobenius algebra and set  notations.

Let $A$ be a finite-dimensional, unital, associative,
and commutative 
algebra over a field $K$. A non-degenerate
bilinear form $\eta:A\tensor A\lrar K$
is a \emph{Frobenius form} if
\be
\label{Frob}
\eta\big(v_1,m(v_2,v_3)\big) 
= \eta\big(m(v_1,v_2),v_3\big),
\qquad v_1,v_2,v_3 \in A,
\ee
where $m:A\tensor A\lrar A$ is the multiplication.
We denote by 
\be
\label{lam}
\lam:A\overset{\sim}{\lrar}
 A^*, \qquad \la\lam(u),v\ra 
=\eta(u,v),
\ee
 the canonical 
isomorphism of the algebra A and its dual.
We assume that $\eta$ is a symmetric bilinear form.
Let $\mathbf{1}\in A$ denote the multiplicative 
identity. Then it defines a \emph{counit}, or 
a \emph{trace}, by
\be
\label{counit}
\epsilon:A\lrar K, \qquad 
\epsilon(v) = \eta(\mathbf{1},v).
\ee
The canonical isomorphism $\lam$ introduces
a unique cocommutative and coassociative
 coalgebra structure in $A$ by the
following commutative diagram.
\be
\label{coproduct}
\begin{CD}
A@>\delta>>A\tensor A\\
@V{\lam}VV  @VV{\lam\tensor\lam}V\\
A^*@>>{m^*}> A^*\tensor A^*
\end{CD}
\ee
It is often convenient to 
use a basis for calculations.  
Let  $\la e_1,e_2,\dots,e_r\ra$ be a $K$-basis for
$A$.
In terms of this basis, the bilinear form $\eta$ is identified with a 
symmetric matrix, and its
inverse is written as follows:
\be
\label{eta-1}
\eta = [\eta_{ij}], \qquad
\eta_{ij}:=\eta(e_i,e_j),
\qquad \eta^{-1}=[\eta^{ij}].
\ee
The comultiplication is then 
written as
$$
\delta(v) = \sum_{i,j,a,b}
 \eta\big(v,m(e_i,e_j)\big) \eta^{ia}\eta^{jb}
 e_a\tensor e_b.
$$
From now on, if there is no confusion, we denote
simply by $m(u,v) = uv$. 
The symmetric Frobenius form and the commutativity
of the multiplication
makes 
\be
\label{symmetric}
\eta\big(e_{i_1}\cdots
e_{i_{j}},e_{i_{j+1}}\cdots e_n\big)
=\epsilon(e_{i_1}\cdots e_{i_n}),
\qquad 1\le j<n,
\ee
completely symmetric with respect to permutations
of the indices. 

The following  is a standard formula
for a non-degenerate bilinear form:
\be
\label{complete set}
v
= \sum_{a,b} \eta(v,e_a)\eta^{ab}e_b.
\ee
It immediately follows that

\begin{lem}
\label{lem:delta m}
 The following diagram commutes:
\be
\label{delta m}
\xymatrix{
&A\tensor A\tensor  A\ar[dr]^{m\tensor id}&\\
A\tensor A \ar[ur]^{id \tensor \delta}\ar[r]^{\;\;\; m} \ar[dr]_{\delta\tensor id}& 
A \ar[r]^{\delta \;\;\;\;} & A\tensor A .\\
&A\tensor A\tensor A \ar[ur]_{id \tensor m}
		}
\ee
Or equivalently, for every $v_1,v_2$ in A, we have
$$
\delta(v_1v_2) 
= (id\tensor m)\big(\delta(v_1),v_2\big)
= (m\tensor id)\big(v_2,\delta(v_1)\big).
$$
\end{lem}

\begin{proof}
Noticing the
commutativity and cocommutativity of $A$, we have
\begin{align*}
\delta(v_1v_2)&=\sum_{i,j,a,b}
\eta(v_1v_2,e_ie_j)\eta^{ia}\eta^{jb}e_a\tensor e_b
\\
&=
\sum_{i,j,a,b}
\eta(v_1e_i,v_2 e_j)\eta^{ia}\eta^{jb}e_a\tensor e_b
\\
&=
\sum_{i,j,a,b,c,d}
\eta(v_1e_i,e_c)\eta^{cd}\eta(e_d,v_2 e_j)
\eta^{ia}\eta^{jb} e_a\tensor e_b
\\
&=
\sum_{i,j,a,b,c,d}
\eta(v_1,e_ie_c)\eta^{cd}\eta^{ia}\eta(e_d v_2, e_j)
\eta^{jb} e_a\tensor e_b
\\
&=\sum_{i,a,c,d}\eta(v_1,e_ie_c)\eta^{cd}\eta^{ia}
e_a\tensor (e_d v_2)
\\
&=(id\tensor m)\big(\delta(v_1),v_2).
\end{align*}
\end{proof}

In the lemma above we consider the composition
$\delta\circ m$. The other order of operations plays
an essential role in 2D TQFT.

\begin{Def}[Euler element]
The \textbf{Euler element} of a Frobenius 
algebra $A$ is defined by
\be\label{Euler}
\mathbf{e}:= m\circ \delta(\mathbf{1}).
\ee
In terms of basis, the Euler element is given by
\be\label{Euler basis}
\mathbf{e} = \sum_{a,b}\eta^{ab} e_ae_b.
\ee
\end{Def}

Another application of \eqref{complete set}
is the following formula that relates 
the multiplication and comultiplication.

\be\label{prod=coprod}
\left(\lam(v_1)\tensor id\right)\delta(v_2)
= v_1v_2.
\ee
This is because
\begin{align*}
\left(\lam(v_1)\tensor id\right)\delta(v_2)
&=
\sum_{a,b,k,\ell}
\left(\lam(v_1)\tensor id\right)
\eta(v_2,e_ke_\ell)\eta^{ka}\eta^{\ell b}
e_a\tensor e_b
\\
&=
\sum_{a,b,k,\ell}
\eta(v_2e_\ell ,e_k)\eta^{ka}
\eta(v_1,e_a) \eta^{\ell b} e_b
\\
&=
\sum_{b,\ell}
\eta(v_1,v_2e_\ell) \eta^{\ell b} e_b
=v_1v_2.
\end{align*}

\section{2D TQFT}
\label{sect:tqft}

The axiomatic formulation of conformal and 
topological quantum field theories was established 
in 1980s. We refer to Atiyah \cite{Atiyah} and
Segal \cite{Segal}. We  consider only
two-dimensional topological quantum field theories
in this paper.
Again for the purpose of setting notations, 
we provide a brief review of the subject in this 
section. We refer to  fundamental literature,
such as \cite{Kock, Teleman}, for more detail of 2D TQFT. 

A 2D TQFT is a symmetric
monoidal functor $Z$ from the cobordism 
category of oriented surfaces (a surface being a 
cobordism of its boundary circles) to 
the monoidal category of finite-dimensional 
vector spaces over a fixed field
$K$ with the operation of tensor products.
The Atiyah-Segal TQFT axioms
automatically make
the vector space 
\be
\label{ZS1}
Z(S^1) = A
\ee
 a unital commutative 
Frobenius algebra over $K$. 

Let $\Sigma_{g,n}$ be an oriented surface of finite
topological type $(g,n)$, i.e., a surface obtained
by removing $n$ disjoint open discs from a 
compact
oriented two-dimensional topological manifold 
of genus $g$. The boundary
components are labeled by 
indices $1, \dots,n$. We always give the induced 
orientation at each boundary circle. 
The TQFT  then assigns to such 
a surface a multilinear map
\be\label{ZSigma}
\Omega_{g,n}\overset{\text{def}}{=}Z(\Sigma_{g,n}): A^{\tensor n}\lrar K.
\ee
If we change the orientation at the $i$-th boundary,
then the $i$-th factor of the tensor product is changed
to the dual space $A^*$.  
Therefore, if we have $k$ boundary circles with
induced orientation and $\ell$ circles
with opposite orientation, then we have
a multi-linear map
$$
\Omega_{g,k,\bar{\ell}}:A^{\tensor k}\lrar
A^{\tensor \ell}.
$$
The sewing axiom of Atiyah \cite{Atiyah}
requires that
$$
\Omega_{g_2,\ell,\bar{n}}
\circ 
\Omega_{g_1,k,\bar{\ell}}
=
\Omega_{g_1+g_2+ \ell-1,k,\bar{n}}
:A^{\tensor k}\lrar A^{\tensor n}.
$$

\begin{figure}[htb]
\label{fig:sewing}
\includegraphics[width=3in]{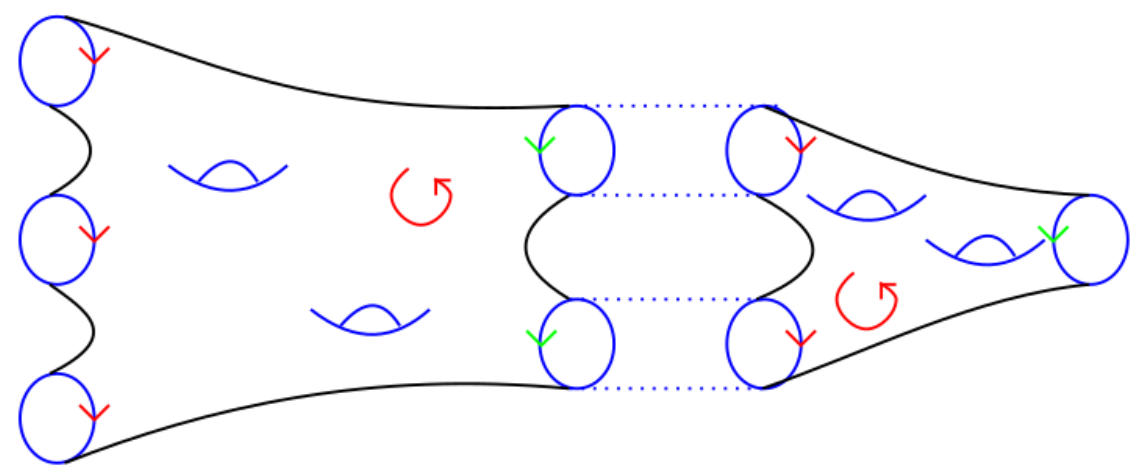}
\caption{}
\end{figure}

A 2D TQFT can be also  obtained as a special case of 
a CohFT of \cite{KM}. 

\begin{Def}[Cohomological Field Theory]
We denote by $\Mbar_{g,n}$ the moduli space
of stable curves of genus $g\ge 0$ and $n\ge 1$ smooth
marked points subject to the stability condition
$2g-2+n>0$. 
Let
\be\label{pi}
\pi:\Mbar_{g,n+1}\lrar \Mbar_{g,n}
\ee
be the forgetful morphism of the last marked point,
and 
\begin{align}
\label{gl1}
&gl_1:\Mbar_{g-1,n+2}\lrar \Mbar_{g,n}
\\
\label{gl2}
&gl_2:\Mbar_{g_1,n_1+1}\times
 \Mbar_{g_2,n_2+1}\lrar \Mbar_{g_1+g_2,n_1+n_2}
\end{align}
the gluing morphisms that give boundary strata of 
the moduli space. An assignment
\be\label{CohFT}
\Omega_{g,n}:A^{\tensor n}\lrar 
H^*(\Mbar_{g,n},K)
\ee
is a CohFT if the following axioms hold:
\begin{align*}
&\textbf{CohFT 0:}\quad \Omega_{g,n} \text{ is 
$S_n$-invariant, i.e., symmetric, and }
\Omega_{0,3}(\mathbf{1},v_1,v_2) = \eta(v_1,v_2).
\\
&\textbf{CohFT 1:} \quad
\Omega_{g,n+1}(v_1,\dots,v_n,\mathbf{1}) =
 \pi^*\Omega_{g,n}(v_1,\dots,v_n).
 \\
&\textbf{CohFT 2:}\quad
gl_1 ^*\Omega_{g,n}(v_1,\dots,v_n)
=\sum_{a,b}\Omega_{g-1,n+2}
(v_1,\dots,v_n,e_a,e_b)\eta^{ab}.
\\
&\textbf{CohFT 3:}\quad
gl_2^*\Omega_{g_1+g_2,|I|+|J|}(v_I,v_J)
=\sum_{a,b}\Omega_{g_1,|I|+1}(v_I,e_a)
\Omega_{g_2,|J|+1}(v_J,e_b)\eta^{ab},
\end{align*}
where $I\sqcup J = \{1,\dots,n\}$. 
\end{Def}

If 
a CohFT takes values  in 
$H^0(\Mbar_{g,n},K) = K$, then it  is a 2D TQFT. 
In what follows, we only consider 
CohFT with values in $H^0(\Mbar_{g,n},K)$.

\begin{rem}
The forgetful morphism makes sense for a 
stable pointed curve, but it does not exist
for a topological surface with boundary
in the same way. Certainly
we cannot just \emph{forget} a boundary. 
For a TQFT, eliminating a boundary corresponds
to capping a disc. In algebraic geometry language,
it is the same as gluing a component of 
$g=0$ and $n=1$. 
Since $H^0(\Mbar_{g,n},K) = K$ is not
affected by the morphism \eqref{pi}-\eqref{gl2},
the equation
$$
\Omega_{g,n}(\mathbf{1},v_2,\dots,v_n) =
 \Omega_{g,n-1}(v_2,\dots,v_n)
$$
is identified with CohFT 3 for $g_2=0$
and $J=\emptyset$,
if we define
\be\label{CohFT01}
\Omega_{0,1}(v):=\epsilon(v) = \eta(\mathbf{1},v),
\ee
even though $\Mbar_{0,1}$ does not exist.
We then have
\begin{align*}
\Omega_{g,n}(v_1,\dots,v_n) 
&=
\sum_{a,b}\Omega_{g,n+1}(v_1,\dots,v_n,e_a)
\eta(\mathbf{1},e_b)\eta^{ab}
\\
&=\Omega_{g,n+1}(v_1,\dots,v_n,\mathbf{1})
\end{align*}
by \eqref{complete set}.
In other words, the  isomorphism
of the degree $0$ cohomologies
\be\label{pi on 0}
\pi^*:H^0(\Mbar_{g,n},K)
\lrar H^0(\Mbar_{g,n+1},K)
\ee
is replaced by its left inverse
\be\label{sigma on 0}
\sigma_i ^*:H^0(\Mbar_{g,n+1},K)
\lrar H^0(\Mbar_{g,n},K),
\ee
where 
\be\label{sigma}
\sigma_i:\Mbar_{g,n}\lrar \Mbar_{g,n+1}
\ee
is one of the $n$ tautological sections.
Of course this consideration does not 
apply for CohFT. 
\end{rem}

\begin{rem}
In the same spirit, although $\Mbar_{0,2}$ does not
exist either, we  can \emph{define}
\be\label{CohFT02}
\Omega_{0,2}(v_1,v_2):=\eta(v_1,v_2)
\ee
so that we  exhaust all  cases appearing in the
Atiyah-Segal axioms for 2D TQFT. In particular,
for $g_2=0$ and $J=\{n\}$, we have
\begin{align*}
\Omega_{g,n}(v_1,\dots,v_n)
&=\Omega_{g,n}\left(
v_1,\dots,v_{n-1},\sum_{a,b}
\eta(v_n,e_b)\eta^{ab}e_a
\right)
\\
&=
\sum_{a,b}\Omega_{g,n}(v_1,\dots, v_{n-1},
e_a) \Omega_{0,2}(v_n,e_b)\eta^{ab}.
\end{align*}
Thus $\Omega_{0,2}(v_1,v_2)$
functions as the identity operator of the 
Atiyah-Segal axiom \cite{Atiyah}.
\end{rem}

\begin{rem}
A marked point $p_i$ of a stable curve $\Sigma
\in \Mbar_{g,n}$ is an insertion point for the
cotangent class $\psi_i = c_1(\bL_i)$, where
$\bL_i$ is the pull-back of the relative
canonical sheaf  on the universal
curve $\pi:\Mbar_{g,n+1}\lrar \Mbar_{g,n}$
by the $i$-th tautological section $\sigma_i:
\Mbar_{g,n}\lrar \Mbar_{g,n+1}$.
If we cut a small disc around $p_i\in \Sigma$,
then the orientation induced on the boundary
circle is consistent with the orientation of the 
unit circle in $T^*_{p_i}\Sigma$.
This orientation is opposite to the orientation
that is 
naturally induced on $T_{p_i}\Sigma$. 
In general, if $V$ is an oriented real vector 
space of dimension $n$, then $V^*$ naturally
acquires the opposite orientation with respect to
the dual basis if $n\equiv 2, 3 \mod 4$.
\end{rem}

As we have noted, in terms of
{sewing axioms}, if 
a boundary circle on 
a topological surface $\Sigma$ of type $(g,n)$
is oriented according to the induced orientation,
then this is an \emph{input} circle to which we
assign an element of $A$. If a boundary circle
is oppositely oriented, then it is an \emph{output}
circle and $\Sigma$ produces an output 
element at this 
boundary. Thus if $\Sigma_1$ has an input
circle and $\Sigma_2$ an output circle, then
we can sew the two surfaces
together along the circle
to form a connected sum $\Sigma_1\#
\Sigma_2$, where the output
from $\Sigma_2$ is placed as  input
 for $\Sigma_1$.

\begin{prop}
The genus $0$ values of a 2D TQFT is given by
\be\label{TQFT 0n}
\Omega_{0,n}(v_1,\dots,v_n) = 
\epsilon(v_1\cdots v_n),
\ee
provided that we \emph{define}
\be\label{TQFT 03}
\Omega_{0,3}(v_1,v_2,v_3):= \epsilon(v_1v_2v_3).
\ee
\end{prop}

\begin{proof}
This is a direct consequence of CohFT 3 and
\eqref{complete set}.
\end{proof}

One of the original motivations of TQFT
\cite{Atiyah, Segal} is to identify the
\emph{topological invariant} $Z(\Sigma)$
of a closed 
manifold $\Sigma$. In our current setting, 
it is defined as 
\be\label{g-invariant}
Z(\Sigma_g):= \epsilon\big(
\lam^{-1}(\Omega_{g,1})\big)
\ee
for a closed oriented surface $\Sigma_g$ of genus
$g$. Here, $\Omega_{g,1}:A\lrar K$ is 
an element of
$A^*$, and $\lam:A\overset{\sim}{\lrar} A^*$
is the canonical isomorphism.

\begin{prop}
\label{prop:g-invariant}
The topological invariant $Z(\Sigma_g)$ of
\eqref{g-invariant} is given by
\be\label{g-invariant formula}
Z(\Sigma_g) = \epsilon(\mathbf{e}^g),
\ee
where $\mathbf{e}^g\in A$ represents the $g$-th power of the Euler element
of \eqref{Euler}.
\end{prop}

\begin{lem}
We have
\be\label{11=Euler}
\mathbf{e} := m\circ \delta(1) 
= \lam^{-1}(\Omega_{1,1}).
\ee
\end{lem}

\begin{proof}
This follows from
$$
\Omega_{1,1}(v) =
\sum_{a,b} \Omega_{0,3}(v,e_a,e_b)\eta^{ab}
=
 \sum_{a,b}\eta(v,e_ae_b)
\eta^{ab}=\eta(v,\mathbf{e})
$$
for every $v\in A$. 
\end{proof}

\begin{proof}[Proof of 
Proposition~\ref{prop:g-invariant}] Since the starting
 case  $g=1$ follows from the above Lemma,
we prove the formula by induction, which goes 
as follows:
\begin{align*}
\Omega_{g,1}(v) 
&= 
\sum_{a,b}
\Omega_{g-1,3}(v,e_a,e_b)\eta^{ab}
\\
&= 
\sum_{i,j,a,b}
\Omega_{0,4}(v,e_a,e_b,e_i)
\Omega_{g-1,1}(e_j)\eta^{ab}\eta^{ij}
\\
&=
\sum_{i,j,a,b}
\eta(ve_ae_b,e_i)
\Omega_{g-1,1}(e_j)\eta^{ab}\eta^{ij}
\\
&=
\sum_{i,j}
\eta(v\mathbf{e},e_i)
\Omega_{g-1,1}(e_j)\eta^{ij}
\\
&=
\Omega_{g-1,1}(v\mathbf{e})
\\
&=
\Omega_{1,1}(v\mathbf{e}^{g-1})
\\
&=
\eta(v\mathbf{e}^{g-1},\mathbf{e})
= \eta(v,\mathbf{e}^g).
\end{align*}
\end{proof}

A closed genus $g$ surface is obtained by
sewing $g$ genus $1$ pieces with one 
output boundaries to a genus $0$ surface with
$g$ input boundaries. Since the Euler element
is the output of the genus $1$ surface
with one boundary, we 
obtain the same result
$$
Z(\Sigma_g) = \Omega_{0,g}(\overset{g}{\overbrace{\mathbf{e},\dots,
\mathbf{e}}}).
$$
Finally we have the following:

\begin{thm}
The value or the 2D TQFT is given by
\be\label{TQFT gn}
\Omega_{g,n}(v_1,\dots,v_n)
=
\epsilon(v_1\cdots v_n \mathbf{e}^g).
\ee
\end{thm}

\begin{proof}
The argument is the same as the proof
of Proposition~\ref{prop:g-invariant}:
\begin{align*}
\Omega_{g,n}(v_1,\dots,v_n)
&=
\Omega_{1,n}(v_1\mathbf{e}^{g-1},v_2,\dots,
v_n)
\\
&=
\sum_{a,b}
\Omega_{0,n+2}(v_1\mathbf{e}^{g-1},v_2,\dots,
v_n, e_a,e_b)\eta^{ab}
\\
&=
\epsilon(v_1\cdots v_n \mathbf{e}^g).
\end{align*}
\end{proof}

\begin{ex} Let $G$ be a finite group.
The center of the complex group algebra
$Z\bC[G]$  is a semi-simple 
Frobenius algebra over $\bC$. 
For every conjugacy class $c$ of $G$, 
the sum of group elements in $c$, 
$$
v(C) := \sum_{u\in C} u \in \bC[G],
$$
is central and defines an element of $Z\bC[G]$.
Although we do not discuss it any further
here, the corresponding TQFT is equivalent to 
 counting problems of 
 character varieties of the fundamental group 
of $n$-punctured topological surface of genus $g$
into $G$. 
\end{ex}

\section{The edge-contraction axioms}
\label{sect:ECA}

In this section we give a   formulation of 
 2D TQFTs
based on the edge-contraction operations on
cell graphs and a new set of 
axioms. The main theorem of this section,
Theorem~\ref{thm:independence}, motivates
our construction of the category of cell graphs and 
 the Frobenius ECO functor in 
 Section~\ref{sect:category}.

\begin{Def}[Cell graphs]
A connected \textbf{cell graph} 
of topological type $(g,n)$ is the
$1$-skeleton
(the union of $0$-cells and
$1$-cells) of a cell-decomposition of a
connected
compact oriented topological 
surface of genus $g$ with 
$n$ labeled $0$-cells. We call a $0$-cell a 
\emph{vertex}, a $1$-cell an \emph{edge}, 
and a $2$-cell a \emph{face}, of a cell graph.
\end{Def}

\begin{rem}
The \emph{dual} of a cell graph is
usually referred to as
a \emph{ribbon graph}, or a
\emph{dessin d'enfant} of
Grothendieck.  A ribbon graph
is a graph
with cyclic order assigned to incident half-edges
at each vertex. Such assignments induce a
cyclic order of  half-edges at each vertex of
the dual graph. Thus
 a cell graph itself
is a ribbon graph.
We note that  vertices
of a cell graph are labeled, which corresponds to 
the usual face labeling of a ribbon graph.
\end{rem}

\begin{rem}We identify two cell graphs if there is a 
homeomorphism of the surfaces that brings 
one cell-decomposition to the other, 
keeping the labeling of $0$-cells. The only 
possible automorphisms of a cell graph
come from cyclic rotations of half-edges
at each vertex. 
\end{rem}

We denote by $\Gam_{g,n}$ the set of 
connected cell graphs of type $(g,n)$ with labeled
vertices.

\begin{Def}[Edge-contraction axioms]
\label{def:ECA}
The \textbf{edge-contraction axioms} are the 
following set of rules for the assignment
\be
\label{Omega}
\Omega:\Gam_{g,n}\lrar (A^*)^{\tensor n}
\ee
of 
a multilinear map
$$
\Omega(\gam):A^{\tensor n}\lrar K
$$
to each cell graph $\gam\in \Gam_{g,n}$.
We consider $\Omega(\gam)$ an $n$-variable function
$\Omega(\gam)(v_1,\dots,v_n)$, where we assign
$v_i\in A$ to the $i$-th vertex of $\gam$.
\begin{itemize}
\item \textbf{ECA 0}: For the simplest cell graph $\gam_0=\bullet \in \Gam_{0,1}$
that consists of only one vertex without any edges,
we define
\be\label{ECA0}
\Omega(\bullet)(v) = \epsilon(v), \qquad v\in A.
\ee

\item \textbf{ECA 1}: Suppose there is an edge $E$
connecting the $i$-th vertex and the $j$-th
vertex for $i<j$ in $\gam\in \Gam_{g,n}$.
Let $\gam'\in \Gam_{g,n-1}$ denote the cell graph
obtained by contracting $E$. Then
\be\label{ECA1}
\Omega(\gam)(v_1,\dots,v_n) = 
\Omega(\gam')(v_1,\dots,v_{i-1},
v_iv_j,v_{i+1},\dots, \widehat{v_j},\dots,v_n),
\ee
where $\widehat{v_j}$ means we omit the 
$j$-th variable $v_j$ at the $j$-th vertex,
which no longer exists in $\gam'$.

\begin{figure}[htb]
\includegraphics[height=0.8in]{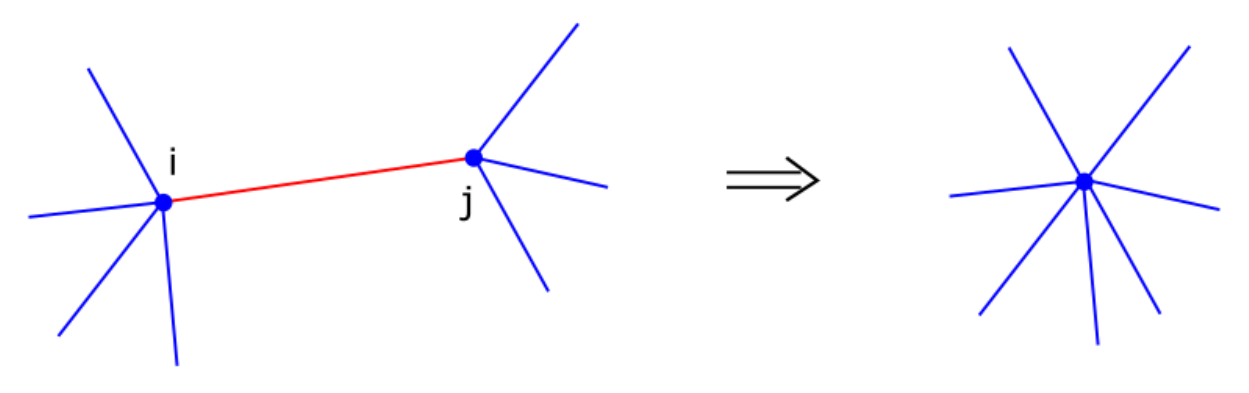}
\caption{The edge-contraction operation that
shrinks a straight edge connecting Vertex $i$
and Vertex $j$.
}
\label{fig:ECA1}
\end{figure}

\item \textbf{ECA 2}: Suppose there is a loop
$L$ in $\gam\in\Gam_{g,n}$ at the $i$-th vertex.
Let $\gam'$ denote the possibly 
disconnected graph obtained by 
contracting $L$ and separating the vertex
to two distinct vertices labeled by $i$ and $i'$.
For the purpose of labeling all vertices, 
we assign an ordering $i-1<i<i'<i+1$.

\begin{figure}[htb]
\includegraphics[height=0.8in]{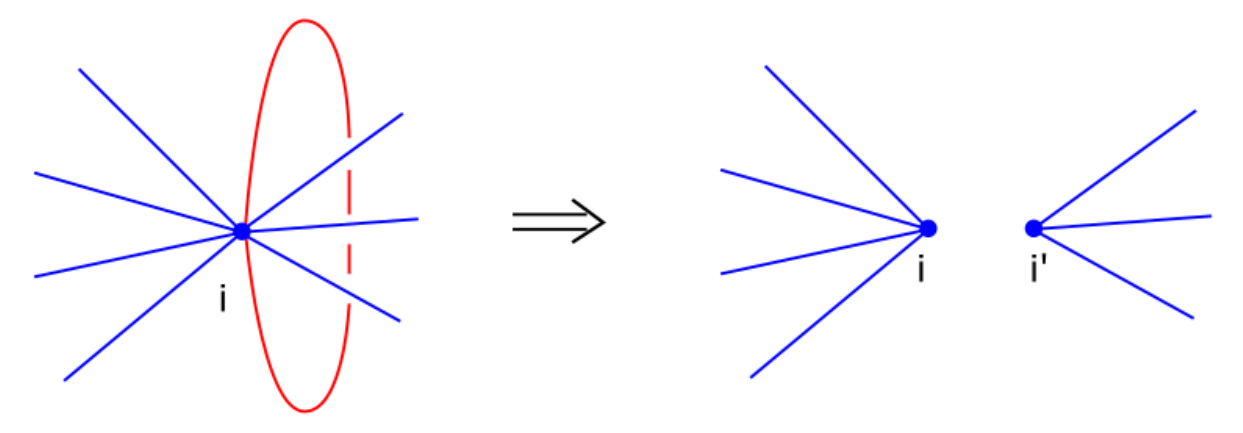}
\caption{The edge-contraction operation that
shrinks a loop attached Vertex $i$.
}
\label{fig:ECA2}
\end{figure}

If $\gam'$ is connected, then it is in $\Gam_{g-1,n+1}$.
We  call  $L$  a \textit{loop of a handle}.
We then impose
\be\label{ECA2-1}
\Omega(\gam)(v_1,\dots,v_n) = 
\Omega(\gam')(v_1,
\dots,v_{i-1},\delta(v_i),v_{i+1},\dots,v_n),
\ee
where the outcome of the comultiplication $\delta(v_i)$
is placed in the $i$-th and $i'$-th slots.

If $\gam'$ is disconnected, then write
$\gam'=(\gam_1,\gam_2)\in \Gam_{g_1,|I|+1}
\times \Gam_{g_2,|J|+1}$, where 
\be\label{disconnected}
\begin{cases}
g=g_1+g_2\\
I\sqcup J = \{1,\dots,\widehat{i},\dots,n\}
\end{cases}.
\ee
In this case  $L$ is a \textit{separating loop}.
Here, vertices labeled by $I$ belong to the connected 
component of genus 
$g_1$, and those labeled by $J$ on the other 
component. Let $(I_-,i,I_+)$ (reps. $(J_-,i,J_+)$) be reordering of $I\sqcup \{i\}$ (resp. $J\sqcup \{i\}$)
in the increasing order. 
We impose
\be\label{ECA2-2}
\Omega(\gam)(v_1,\dots,v_n) = \sum_{a,b,k,\ell}
\eta(v_i,e_ke_\ell)\eta^{ka}\eta^{\ell b}
\Omega(\gam_1)(v_{I_-},e_a,v_{I_+})
\Omega(\gam_2)(v_{J_-},e_b,v_{J_+}),
\ee
which is  similar to \eqref{ECA2-1}, 
just the comultiplication $\delta(v_i)$ is
written in terms of the basis. Here,
cocommutativity of $A$ is assumed in this formula.
\end{itemize}

\end{Def}

\begin{rem}
We do not assume the permutation symmetry 
of $\Omega(\gam)(v_1,\dots,v_n)$. The cumbersome 
notation of the axioms is due to keeping track of
the ordering of indices. 
\end{rem}

\begin{rem}
Let us define $m(\gam)=2g-2+n$
for $\gam\in \Gamma_{g,n}$. The
edge-contraction operations are reduction of 
$m(\gam)$ exactly by $1$. 
Indeed, for ECA 1, we have 
$$
m(\gam') = 2g -2
+(n-1) = m(\gam)-1.
$$
ECA 2 applied to a loop of a handle produces
$$
m(\gam') = 2(g-1)-2+(n+1) = m(\gam)-1.
$$ 
For a separating loop, we have
$$
\begin{matrix}
&2g_1-2+|I|+1
\\
{+)}&{2g_2-2+|J|+1}
\\
&\overline{2g_1+2g_2-4+|I|+|J|+2}
&=\; \;2g-2+n-1.
\end{matrix}
$$
\end{rem}

This reduction  is used in the
proof of the following theorem.

\begin{thm}[Graph independence]
\label{thm:independence}
As the consequence of the edge-contraction axioms,
every connected cell graph $\gam\in \Gam_{g,n}$
gives rise to the same map
\be\label{independence}
\Omega(\gam): A^{\tensor n}\owns 
v_1\tensor \cdots \tensor v_n \longmapsto
\epsilon(v_1\cdots v_n \mathbf{e}^g)\in  K,
\ee
where $\mathbf{e}$ is the Euler element 
of \eqref{Euler}.
In particular, $\Omega(\gam)(v_1,\dots,v_n)$ is 
symmetric with respect to permutations of indices.
\end{thm}

\begin{cor}[ECA implies TQFT]
\label{cor:ECA=TQFT}
Define $\Omega_{g,n}(v_1,\dots,v_n) = 
\Omega(\gam)(v_1,\dots,v_n)$ for any
$\gam\in \Gam_{g,n}$. Then 
$\{\Omega_{g,n}\}$ is the 2D TQFT
associated with the Frobenius algebra $A$. 
Every 2D TQFT is obtained in this way,  hence 
 the two descriptions of 2D TQFT
are equivalent.
\end{cor}

\begin{proof}[Proof of Corollary~\ref{cor:ECA=TQFT}
assuming Theorem~\ref{thm:independence}]
Since both ECAs and 2D TQFT give the 
unique value
$$
\Omega(\gam)(v_1,\dots,v_n)
=\epsilon(v_1\cdots v_n\mathbf{e}^g)
=\Omega_{g,n}(v_1,\dots,v_n)
$$
for all $(g,n)$ from \eqref{TQFT gn}, we 
see that the two sets of 
axioms are equivalent, and also
that the edge-contraction axioms 
produce evert 2D TQFT. 
\end{proof}

To illustrate the graph independence, 
let us first examine three simple cases.

\begin{lem} [Edge-removal lemma]
\label{lem:reduced}
Let $\gam\in \Gam_{g,n}$.
\begin{enumerate}
\item Suppose there is a disc-bounding loop $L$ in 
$\gam$ (the  graph on the left
of Figure~\ref{fig:eliminate}). 
Let $\gam'\in \Gam_{g,n}$ be the graph
obtained by removing $L$ from $\gam$. 

\item Suppose there are two edges $E_1$ and
$E_2$ between 
two distinct vertices Vertex $i$ and Vertex $j$,
$i < j$, that bound a disc
(the middle graph 
of Figure~\ref{fig:eliminate}). Let $\gam'\in \Gam_{g,n}$ 
be the graph obtained by 
removing $E_2$.

\item Suppose two loops, $L_1$ and $L_2$, are
attached to the $i$-th vertex
(the  graph on the right
of Figure~\ref{fig:eliminate}). If they are 
homotopic, then let $\gam'\in \Gam_{g,n}$
be the graph obtained by removing $L_2$ from 
$\gam$. 
\end{enumerate}
In each of the above cases, we  have
\be\label{gam=gam'}
\Omega(\gam)(v_1,\dots,v_n) = 
\Omega(\gam')(v_1,\dots,v_n).
\ee
\end{lem}

\begin{figure}[htb]
\includegraphics[height=1in]{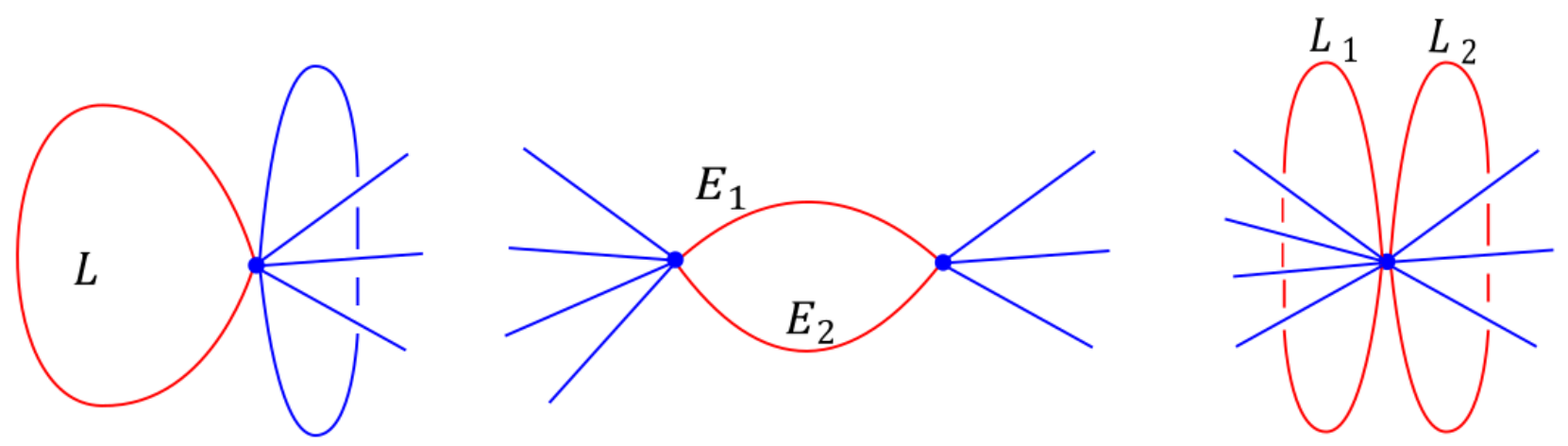}
\caption{
}
\label{fig:eliminate}
\end{figure}

\begin{proof}
(1)  Contracting a disc-bounding loop attached to 
the $i$-th vertex creates 
$(\gam_0,\gam')\in \Gam_{0,1}\times \Gam_{g,n}$,
where $\gam_0$ consists of only one vertex and
no edges. Then ECA 2 reads
\begin{align*}
\Omega(\gam)(v_1,\dots,v_n) &= 
\sum_{a,b,k,\ell}\eta(v_i,e_ke_\ell)
\eta^{ka}\eta^{\ell b}\gam_0(e_a)
\Omega(\gam')(v_1,\dots,v_{i-1},e_b,v_{i+1}
\dots,v_n)
\\
&=
\sum_{a,b,k,\ell}\eta(v_i,e_ke_\ell)
\eta^{ka}\eta^{\ell b}\eta(1,e_a)
\Omega(\gam')(v_1,\dots,v_{i-1},e_b,v_{i+1}
\dots,v_n)
\\
&=
\sum_{b,k,\ell}\eta(v_i,e_ke_\ell)
\delta^k _1\eta^{\ell b}
\Omega(\gam')(v_1,\dots,v_{i-1},e_b,v_{i+1}
\dots,v_n)
\\
&=
\sum_{b,\ell}\eta(v_i,e_\ell)
\eta^{\ell b}
\Omega(\gam')(v_1,\dots,v_{i-1},e_b,v_{i+1}
\dots,v_n)
\\
&=
\Omega(\gam')(v_1,\dots,v_{i-1},v_i,v_{i+1}
\dots,v_n).
\end{align*}

(2) Contracting Edge $E_1$ makes $E_2$
 a disc-bounding 
loop at Vertex $i$. We can remove it by
(1). Note that the new Vertex $i$ is assigned
with $v_iv_j$. Restoring $E_1$  makes
the graph exactly the one obtained by removing
 $E_2$ from $\gam$. Thus \eqref{gam=gam'}
holds. 

\smallskip

(3) Contracting Loop $L_1$ makes $L_2$
a disc-bounding loop. Hence we can remove it
by (1). Then restoring $L_1$ creates a graph
obtained from $\gam$ by removing $L_2$.
Thus \eqref{gam=gam'}
holds. 
\end{proof}

\begin{rem}
The three cases
treated above correspond to eliminating
degree $1$ and $2$ vertices from the  ribbon
graph dual to the cell graph. In combinatorial 
moduli theory, we normally consider 
ribbon graphs that have no vertices 
of degree less than 3 \cite{MP1998}. 
\end{rem}

\begin{Def}[Reduced graph]
We call a cell graph \textbf{reduced} if it
 does not have any disc-bounding
loops or disc-bounding bigons. In other words,
the dual ribbon graph of a reduced cell graph 
has no vertices of degree $1$ or $2$.
\end{Def}

We can see from Lemma~\ref{lem:reduced} (1)
that
every $\gam_{0,1}\in\Gam_{0,1}$ gives rise to the 
same map 
\be\label{01}
\Omega(\gam_{0,1})(v) = \epsilon(v). 
\ee
Likewise, 
Lemma~\ref{lem:reduced} (1) and (2) show 
that every $\gam_{0,2}\in \Gam_{0,2}$ gives the
same map
$$
\Omega(\gam_{0,2})(v_1,v_2) = \eta(v_1,v_2).
$$
This is because we can remove all edges and loops but
one that connects the two vertices, and from ECA 1,
the value of the assignment is $\epsilon(v_1v_2)$.

\begin{proof}[Proof of Theorem~\ref{thm:independence}]
We use the induction on $m=2g-2+n$. The base
case is $m=-1$, or $(g,n) = (0,1)$, for which
the theorem holds by \eqref{01}.
Assume that \eqref{independence}
holds for all $(g,n)$ with
$2g-2+n<m$. Now let $\gam
\in \Gam_{g,n}$ be a cell graph 
of type $(g,n)$  such that 
$2n-2+n=m$. 

Choose an
arbitrary straight edge of $\gam$ that connects
two distinct vertices, say Vertex $i$ and Vertex $j$,
$i<j$. By contracting this edge, we obtain by ECA 1,
$$
\Omega(\gam)(v_1,\dots,v_n)  = 
\Omega(\gam_{g,n-1})(v_1,\dots,v_{i-1},v_iv_j,
v_{i+1} \dots, \widehat{v_j},\dots,v_n)
=\epsilon(v_1\dots v_n \mathbf{e}^g).
$$
If we have chosen an arbitrary
loop attached to Vertex $i$, then its contraction by
ECA 2 gives two cases, depending  on whether
the loop is a loop of a handle, or a separating 
loop. For the first case, by appealing to
\eqref{complete set} and
\eqref{Euler basis}, we obtain
\begin{align*}
\Omega(\gam)(v_1,\dots,v_n)  
&= 
\sum_{a,b,k,\ell}\eta(v_i,e_ke_\ell)\eta^{ka}
\eta^{\ell b}
\Omega(\gam_{g-1,n+1})(v_1,\dots,v_{i-1},
e_a,e_b,v_{i+1},\dots,v_n)
\\
&=
\sum_{a,b,k,\ell}\eta(v_ie_k,e_\ell)\eta^{ka}
\eta^{\ell b}
\Omega(\gam_{g-1,n+1})(v_1,\dots,v_{i-1},
e_a,e_b,v_{i+1},\dots,v_n)
\\
&=
\sum_{a,k}\eta^{ka}
\Omega(\gam_{g-1,n+1})(v_1,\dots,v_{i-1},
e_a,v_ie_k,v_{i+1},\dots,v_n)
\\
&=
\sum_{a,k}\eta^{ka}
\epsilon(v_1\cdots v_n \mathbf{e}^{g-1}e_ae_b )
\\
&=
\epsilon(v_1\cdots v_n \mathbf{e}^g).
\end{align*}
For the case of a separating loop, 
again by appealing to \eqref{complete set}, 
we have
\begin{align*}
\Omega(\gam)(v_1,\dots,v_n)  
&= 
\sum_{a,b,k,\ell}\eta(v_i,e_ke_\ell)\eta^{ka}
\eta^{\ell b}
\Omega(\gam_{g_1,|I|+1})
\big(v_{I_-},e_a,v_{I_+}\big)
\Omega(\gam_{g_2,|J|+1})
\big(v_{J_-},e_b,v_{J_+}\big)
\\
&=
\sum_{a,b,k,\ell}\eta(v_i,e_ke_\ell)\eta^{ka}
\eta^{\ell b}
\epsilon\left(e_a \prod_{c\in I}v_c
\mathbf{e}^{g_1}\right)
\epsilon\left(e_b \prod_{d\in J}v_d
\mathbf{e}^{g_2}\right)
\\
&=
\sum_{a,b,k,\ell}\eta(v_ie_k,e_\ell)\eta^{ka}
\eta^{\ell b}
\eta\left(\prod_{c\in I}v_c, e_a
\mathbf{e}^{g_1}\right)
\epsilon\left(e_b \prod_{d\in J}v_d
\mathbf{e}^{g_2}\right)
\\
&=
\sum_{a,k}\eta^{ka}
\eta\left(\prod_{c\in I}v_c\mathbf{e}^{g_1}
, e_a\right)
\epsilon\left(v_ie_k \prod_{d\in J}v_d
\mathbf{e}^{g_2}\right)
\\
&=
\epsilon\left(v_i 
\prod_{c\in I}v_c\mathbf{e}^{g_1}
\prod_{d\in J}v_d
\mathbf{e}^{g_2}\right)
\\
&=
\epsilon(v_1\cdots v_n\mathbf{e}^{g_1+g_2}).
\end{align*}
Therefore, no matter how we apply 
ECA 1 or ECA 2, we always obtain the same
result. This completes the proof.
\end{proof}

\begin{rem}
There is a different proof of the graph independence
theorem, using a topological idea of deforming
graphs
similar to the one used in \cite{MY}.
\end{rem}

As we see, the key reason for
the graph independence of 
Theorem~\ref{thm:independence}
is the 
property of the Frobenius algebra $A$
that we have, namely, commutativity, 
cocommutativity, associativity,
coassociativity, and the Frobenius relation
\eqref{Frob}. These properties are manifest in 
the following graph operations.
Although the next proposition is an 
easy consequence of Theorem~\ref{thm:independence},
we derive it directly from the 
ECAs so that we can see how the 
algebraic structure of the Frobenius algebra is 
encoded into the TQFT. 
Indeed, the graph-independence theorem
also follows from 
Proposition~\ref{prop:commutativity}.
This fact motivates us to introduce the
category of cell graphs and the 
Frobenius ECO functor in the next section.

\begin{prop} [Commutativity of Edge Contractions]
\label{prop:commutativity}
Let $\gam\in\Gam_{g,n}$.
\begin{enumerate}
\item
Suppose Vertex $i$ is connected to two distinct
vertices Vertex $j$ and Vertex $k$ by two edges, 
$E_j$ and $E_k$.
 The graph we obtain, denoted as
$\gam'\in \Gam_{g,n-2}$, by
first contracting $E_j$ and then contracting
$E_k$, is the same as contracting the edges in the
opposite order. The two different orders 
of the application of ECA 1 then gives the same
answer. For example, if $i<j<k$, then we have
\be\label{edge order}
\Omega(\gam)(v_1,\dots,v_n) = 
\Omega(\gam')(v_1,
\dots,v_{i-1},v_iv_jv_k,
v_{i+1},\dots, \widehat{v_j},\dots,
\widehat{v_k},\dots,v_n).
\ee

\item Suppose two loops $L_1$ and $L_2$
are connected to Vertex $i$. Then the contraction 
of the two loops in different orders gives the same
result.

\item Suppose  a loop $L$ and a 
straight edge $E$ are attached to Vertex $i$, 
where $E$ connects to Vertex $j$, $i\ne j$. 
Then contracting $L$ first and followed by 
contracting $E$, gives the same result as we
contract $L$ and $E$ in the other way around.
\end{enumerate}
\end{prop}

\begin{proof}
(1) There are three possible cases: $i<j<k$,
$j<i<k$, and $j<k<i$. In each case, the 
result is replacing $v_i$ by $v_iv_jv_k$, and
removing two vertices. The associativity 
and commutativity of the
multiplication of $A$ make the result of different
contractions the same.

\smallskip
(2) There are two cases here: After the contraction 
of one of the loops, (a) the other loop remans to be a
loop, or (b) becomes an edge connecting the two 
vertices created by the contraction of the first loop.

In the first case (a), the contraction of the two 
loops makes Vertex $i$ in $\gam$ into
three different vertices $i_1,i_2,i_3$ of the
resulting graph $\gam'$, which may be disconnected.
The loop contractions in the two different orders 
produce triple tensor products
$$
(1\tensor \delta)\delta(v_i) = (\delta\tensor 1)\delta(v_i),
$$
which are equal by the coassociativity
\begin{equation*}
		\xymatrix{
&A\tensor A \ar[dr]^{1\tensor \delta}&\\
A\ar[ur]^{\delta} \ar[dr]_{\delta}& & 
A\tensor A\tensor A .\\
		&A\tensor A\ar[ur]_{\delta\tensor 1}
		}
\end{equation*} 
For (b), the contraction of the loops in either order will produce
$m\circ \delta(v_i)$ on the same $i$-th slot of
the same graph $\gam'\in \Gam_{g-1,n}$.

\smallskip
(3) This amounts to proving the equation
$$
\delta(v_iv_j) = (1\tensor m)\big(\delta(v_i),v_j\big)
= (m\tensor 1)\big(v_j,\delta(v_i)\big),
$$
which is Lemma~\ref{lem:delta m}.
\end{proof}

\begin{rem}
If we have a system of 
subsets $\Gam'_{g,n}\subset \Gam_{g,n}$ 
for all $(g,n)$
that is closed under the edge-contraction operations,
then all statements of this section still hold by
replacing $\Gam_{g,n}$ by $\Gam'_{g,n}$.  
\end{rem}

\begin{rem}
Chen \cite{Chen} proved the graph independence
for a special case of $A = Z\bC[S_3]$, the center
of the group algebra for symmetric group $S_3$, by
direct computation. This result led the authors to 
find a general proof of 
Theorem~\ref{thm:independence}.
\end{rem}

The edge-contraction operations are associated
with gluing morphisms of $\Mbar_{g,n}$ 
that are different
from those  in \eqref{gl1} and \eqref{gl2}.
ECA 1 of \eqref{ECA1} is associated
with
\be
\label{alpha}
\a : \Mbar_{0,3}\times \Mbar_{g,n-1}\lrar
\Mbar_{g,n}.
\ee
The handle cutting case of ECA 2 of \eqref{ECA2-1}
is associated with
\be
\label{beta1}
\b_1 : \Mbar_{0,3}\times \Mbar_{g-1,n+1}
\lrar \Mbar_{g,n},
\ee
and the separating loop contraction with
\be
\label{beta2}
\b_2 : \Mbar_{0,3}\times \Mbar_{g_1,|I|+1}
\times \Mbar_{g_2,|J|+1}
\lrar \Mbar_{g_1+g_2,|I|+|J|+1}.
\ee
Although there are no cell graph operations that
are directly associated with the forgetful morphism
$\pi$ and the gluing maps $gl_1$ and $gl_2$, 
there is an operation on cell graphs similar to the
\emph{connected sum}
of topological surfaces. 

\begin{Def}[Connected sum of cell graphs]
Let $\gam'$ be a  cell graph
with the following conditions.
\begin{enumerate}
\item There is a vertex $q$ in $\gam'$ of degree
$d$.
\item There are $d$ distinct edges incident to $q$. 
In particular, none of them is a loop.
\item There are exactly $d$ faces in $\gam'$
incident to $q$.
\end{enumerate}

 Given an arbitrary cell graph $\gam$ with a degree
 $d$ vertex $p$, 
 we can create a
new cell graph $\gam \#_{(p,q)} \gam'$, which we
call the \emph{connected sum}
of $\gam$ and $\gam'$. 
The procedure is the following. We label 
all half-edges incident to $p$ with 
$\{1,2,\dots,d\}$ according to the 
cyclic order of the cell graph $\gam$ at $p$. 
We also label all edges incident to $q$ in 
$\gam'$ with $\{1,2,\dots,d\}$, but this time
opposite to the cyclic oder given to 
$\gam'$ at $q$. Cut a small disc around
$p$ and $q$, and connect all half-edges 
according to the labeling. The result is a
cell graph $\gam \#_{(p,q)} \gam'$. 
\end{Def}

\begin{rem} The connected sum construction can
be applied to
two distinct vertices $p$ and $q$ of the same
graph,  provided
that these vertices satisfy the required conditions.
\end{rem}

\begin{rem}
The total number of vertices decreases by $2$ 
in the connected sum. Therefore, two $1$-vertex
graphs cannot be connected by this construction.
\end{rem}

The connected sum construction provides
the inverse of the edge-contraction operations
as the following diagrams show. It is also 
clear from these figures that the edge-contraction 
operations are degeneration of curves producing 
a rational curve with three special points,
as indicated in Introduction.

\begin{figure}[htb]
\includegraphics[height=0.8in]{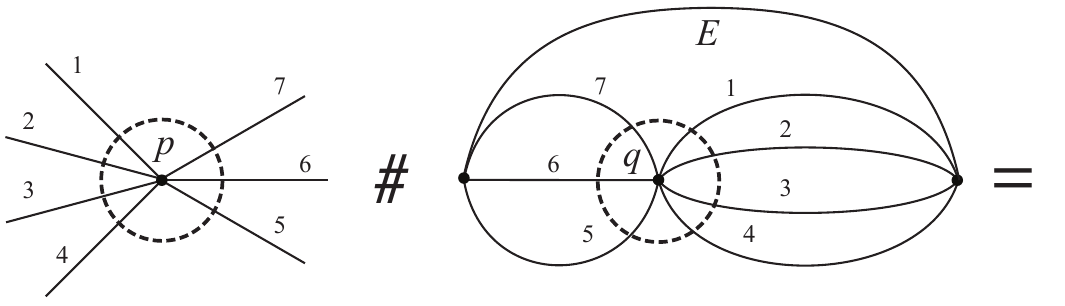}
\includegraphics[height=0.8in]{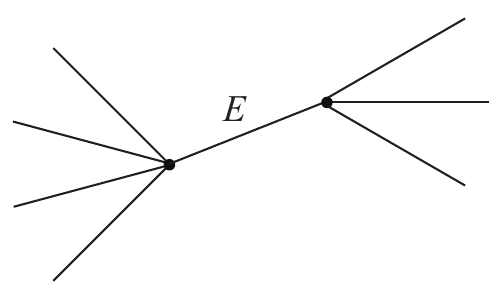}
\caption{The connected sum of a cell graph with 
a particular type $(0,3)$ cell graph  gives the inverse 
of the edge-contraction operation on $E$
 that connects two 
distinct vertices. The connected sum
with the $(0,3)$ piece has to be done so that the
edges incidents on each side of $E$ match
the original graph.
}
\end{figure}

\begin{figure}[htb]
\includegraphics[height=1in]{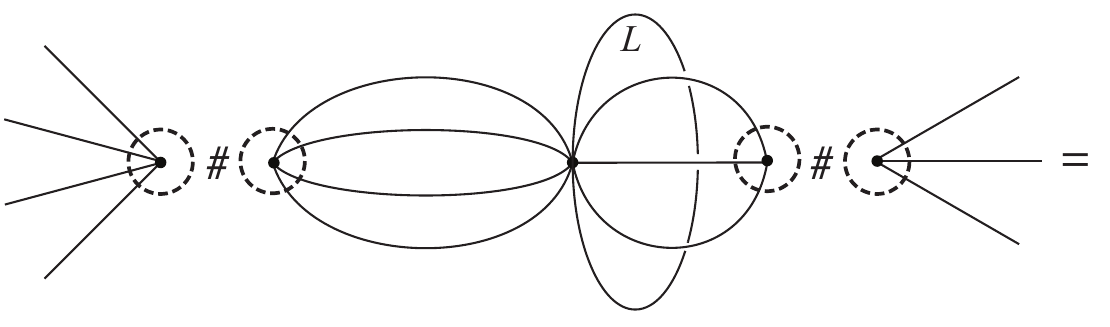}
\includegraphics[height=1in]{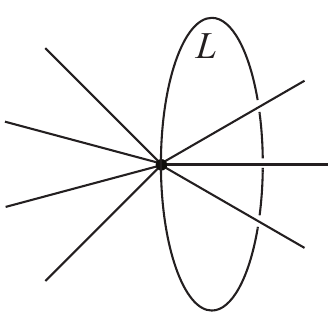}
\caption{The edge-contraction operation on
a loop $L$ 
is the inverse of two connected sum operations,
with a type $(0,3)$ piece in the middle.}
\end{figure}


\section{Category of cell graphs
and Frobenius ECO functors}
\label{sect:category}

In the previous section, we started from 
a Frobenius algebra $A$ and constructed 
the corresponding TQFT  through
edge-contraction axioms. The key step is the
assignment of the linear 
map $\Omega(\gam):A^{\tensor n}
\lrar K$ to each cell graph 
$\gam\in\Gam_{g,n}$. 
As we have noticed, edge-contraction 
operations encode the structure of 
a Frobenius algebra. These considerations 
suggest that cell graphs are functors, and 
edge-contraction operations are natural 
transformations. 
In this section,
we define the category of cell graphs, 
and define Frobenius ECO functors, which make
edge-contraction operations correspond
to natural transformations.

Let $(\cC,\tensor,K)$ be a monoidal 
category with a bifunctor 
$\tensor: \cC\times \cC \lrar \cC$ and its
left and right identity object  $K\in Ob(\cC)$. 
The example we keep in mind is 
the monoidal category $(\Vect,\tensor, K)$
 of vector spaces
defined over a field $K$ with the vector space
tensor product operation. 
Fore brevity, we call the bifunctor 
$\tensor$ just as a tensor product.
A $K$-object in 
$\cC$ is a pair $(V,f:V\lrar K)$ consisting
of an object $V$ and a morphism 
$f:V\lrar K$. We denote by $\cC/K$ the category
of $K$-objects in $\cC$. A $K$-morphism
$h:(V_1,f_1) \lrar (V_2,f_2)$ is 
a morphism $h:V_1\lrar V_2$ in $\cC$
that satisfies the commutativity
\be\label{K-morphism}
\begin{CD}
V_1@>f_1>> K\\
@VhVV @|\\
V_2@>>f_2> K
\end{CD}.
\ee
We note that every morphism  
$h:V_1\lrar V_2$ in $\cC$
yields a new object $(V_1,f_1)$
from a given $(V_2,f_2)$ as in 
\eqref{K-morphism}. 
This is  the \emph{pull-back}
object. 
The category $\cC/K$ itself is  
a monoidal category
with  respect to the tensor product, 
and the final object 
 $(K,id_K:K\lrar K)$ of $\cC/K$
as its identity object.

We denote by $\cF un(\cC/K,\cC/K)$ 
the \textbf{endofunctor category},
 consisting of monoidal functors 
$$
\a:\cC/K \lrar \cC/K
$$ 
as its objects. Let $\a$ and $\b$ be two endofunctors,
and $\tau$ a natural transformation between them.
Natural transformations form  morphisms
in the endofunctor category. 
$$
\xymatrix{
V\ar[dd]_h\ar[dr]^f && \a(V) 
\ar[dd]_{\a(h)}\ar[dr]^{\a(f)} \ar[rr]^\tau 
&&\b(V)\ar[dd]
\ar[dr]^{\b(f)}
\\
&K&&K\ar[rr]^{\!\!\!\!\!\!\!\!\!\!\!\!\!\!\!\!\tau}&&K
\\
W \ar[ur]_g &&\a(W) \ar[ur]_{\a(g)} 
\ar[rr]_\tau&&\b(W)\ar[ur]_{\b(g)}
}
$$
The final object of
$\cF un(\cC/K,\cC/K)$ is the functor 
\be\label{phi}
\phi: (V,f:V\lrar K) \lrar (K,id_K:K\lrar K)
\ee
which assigns the final object of the codomain
$\cC/K$ to 
everything in the domain $\cC/K$. 
With respect to the tensor product and the above
functor $\phi$ as its identity object,
the endofunctor category $\cF un(\cC/K,\cC/K)$
is again a monoidal category.

\begin{Def}[Subcategory generated by $V$]
\label{def:TV}
For every choice of an object $V$ of $\cC$,
we define a category of $K$-objects
$\cT(V^\bullet)/K$  as the 
 full subcategory of $\cC/K$
 whose objects are
$(V^{\tensor n},f:V^{\tensor n}\lrar K)$, 
$n=0,1,2,\dots .$ We call $\cT(V^\bullet)/K$
the \textbf{subcategory generated by} $V$ in 
$\cC/K$.
\end{Def}

\begin{Def}[Monoidal category of cell graphs]
\label{def:CG}
The  finite coproduct (or cocartesian) 
\textbf{monoidal
category of cell graphs} 
$\cC\cG$ is defined as follows. 
\begin{itemize}
\item The set of objects $Ob(\cC\cG)$ consists of a
finite disjoint union of cell graphs. 
\item The coproduct in $\cC\cG$ is the 
disjoin union, and the coidentity object is the empty
graph.
\end{itemize}
 The set of morphism $\Hom(\gam_1,\gam_2)$
from a cell graph $\gam_1$ to $\gam_2$ consists
of equivalence classes of sequences of 
edge-contraction operations and graph automorphisms.
For brevity of notation, if $E$ is an edge
connecting two distinct vertices of $\gam_1$, then
we simply denote by $E$ itself as the 
edge-contraction operation shrinking $E$, 
as in Figure~\ref{fig:ECA1}. If $L$
is a loop in $\gam_1$, then we denote by 
$L$ the edge-contraction operation 
of Figure~\ref{fig:ECA2}. 
Let
$$
\widetilde{\Hom}(\gam_1,\gam_2)
=\left\{\begin{matrix}
\text{composition of 
a sequence of edge-contractions}\\
\text{ 
and graph automorphisms that changes 
$\gam_1$ to $\gam_2$}
\end{matrix}
\right\}.
$$
This is the set of words consisting of
edge-contraction operations and graph 
automorphisms that change $\gam_1$
to $\gam_2$ when operated consecutively. 
If there is no such operations,
then we define $\widetilde{\Hom}(\gam_1,\gam_2)$
to be the empty set.
The morphism set 
$\Hom(\gam_1,\gam_2)$ is the 
set of equivalence
classes of $\widetilde{\Hom}(\gam_1,\gam_2)$.
The equivalence relation in 
the extended morphism set is generated
by the following cases of equivalences.

\begin{enumerate}

\item Suppose $\gam_1$ has a non-trivial 
automorphism $\sigma$. Then for every 
edge $E$ of $\gam_1$, $E$ and $\sigma(E)$
are equivalent.

\item
Suppose Vertex $i$ of 
$\gam_1\in \Gam_{g,n}$
 is connected to two distinct
vertices Vertex $j$ and Vertex $k$ by two edges, 
$E_j$ and $E_k$.
 The graph we obtain, denoted as
$\gam_2\in \Gam_{g,n-2}$, by
first contracting $E_j$ and then contracting
$E_k$, is the same as contracting the edges in the
opposite order. The two words $E_1E_2$ and
$E_2E_1$ are equivalent.

\item Suppose two loops $L_1$ and $L_2$
of $\gam_1$ 
are connected to Vertex $i$. Then the contraction 
operations 
of the two loops in different orders give the same
result. The two words 
$L_1L_2$ and
$L_2L_1$ are equivalent.

\item Suppose  a loop $L$ and a 
straight edge $E$ in $\gam_1$ are 
attached to Vertex $i$, 
where $E$ connects to Vertex $j$, $i\ne j$. 
Then contracting $L$ first and followed by 
contracting $E$, gives the same result as we
contract $L$ and $E$ in the other way around.
The two words $EL$ and $LE$ are
equivalent.

\item Suppose $\gam_1$ has two edges 
(including loops) $E_1$ and $E_2$ that have
no common vertices, and $\gam_2$ is obtained 
by contracting them. Then $E_1E_2$ is equivalent
to $E_2E_1$. 

\item Suppose two edges $E_1$ and $E_2$
are both incident to two distinct vertices.
Then $E_1E_2$ is equivalent to
$E_2E_1$. 
\end{enumerate}
\end{Def}

\begin{ex}
A few simple examples of morphisms 
are given below.   
\begin{align*}
\Hom(\bullet\!\! \frac{\;E_1\;}{}\!\!\!\bullet
\!\! \frac{E_2\;\;}{}\!\!\!\bullet, 
\bullet\!\! \frac{\phantom{\;E_1\;}}{}\!\!\!\bullet)
&=
\{E_1,E_2\},
\\
\Hom(\bullet\!\! \frac{\;E_1\;}{}\!\!\!\bullet
\!\! \frac{E_2\;\;}{}\!\!\!\bullet, 
\bullet)
&=
\{E_1E_2\},
\\
\Hom\left(
\underset{E_2}{\overset{E_1}{\bullet\!\!\!\bigcirc\!\!\!\bullet}},\bullet\!\!\bigcirc\right)
&= \{E_1\} = \{\sigma(E_1)\} = \{E_2\},
\\
\Hom\left(
\underset{E_2}{\overset{E_1}{\bullet\!\!\!\bigcirc\!\!\!\bullet}},\bullet\;\;\bullet\right)
&= \{E_1E_2\}.
\end{align*}
\end{ex}

The cell graph on the left of the third and fourth lines 
has an automorphism $\sigma$
that interchanges $E_1$
and $E_2$. Thus as the edge-contraction operation,
$E_2=E_1\circ \sigma = \sigma(E_1)$.

\begin{rem}
If $\gam\in \Gam_{g,n}$, then 
$\Hom(\gam,\gam) = \{id_\gam\}$.
\end{rem}

We have seen in the last section that
when we have made a choice of a
unital commutative Frobenius algebra
$A$, a cell graph $\gam\in \Gam_{g,n}$
defines a multilinear map
$\Omega_A(\gam):A^{\tensor n} \lrar K$ subject to 
edge-contraction axioms. 
For a different Frobenius algebra $B$, 
we have a different multilinear map 
$\Omega_B(\gam):B^{\tensor n} \lrar K$,
subject to the same axioms. These two maps are
unrelated, unless we have a Frobenius
algebra homomorphism $h:A\lrar B$. 
Theorem~\ref{thm:independence} 
tells us that we have a $K$-morphism
of \eqref{K-morphism} which induces
$\Omega_A(\gam)$ as the pull-back of
$\Omega_B(\gam)$.
$$
\xymatrix{
A\ar[d]_h & A^{\tensor n} \ar[d] 
\ar[rr]^{\Omega_A(\gam)} 
&&K\ar[d]
\\
B  & B^{\tensor n} 
\ar[rr]_{\Omega_B(\gam)}&& K
}
$$
This consideration suggests that $\Omega(\gam)$
is a functor defined on the category of 
Frobenius algebras. But since we are 
encoding the Frobenius algebra structure
into the category of cell graphs, the 
extra choice of
Frobenius algebras is redundant.

We are thus led to the following definition.

\begin{Def}[Frobenius ECO functor]
\label{def:Frobfunctor}
An \textbf{Frobenius ECO functor}  is 
a   monoidal  functor 
\be\label{omega}
\omega:\cC\cG \lrar \cF un(\cC/K,  \cC/K)
\ee
satisfying the following conditions.
\begin{itemize}
\item The graph $\gam_0=\bullet$ of \eqref{ECA0}
of type $(0,1)$ consisting of only one vertex
and no edges corresponds to the identity 
endofunctor:
\be\label{bullet functor}
\bullet \lrar (id:\cC/K\lrar \cC/K).
\ee
\item A graph $\gam\in \Gam_{g,n}$ of type 
$(g,n)$ corresponds to a functor
\be\label{type gn}
\gam \longmapsto \left[
(V,f:V\lrar K) \lrar  (V^{\tensor n},
\omega_V(\gam):V^{\tensor n}\lrar K)\right].
\ee
\end{itemize}
The Frobenius ECO functor assigns to each 
edge-contraction operation 
a natural transformation of  
endofunctors $\cC/K \lrar  \cC/K$.
\end{Def}

\begin{rem}
The unique
construction of the Frobenius ECO functor
for $(\Vect,\tensor, K)$ requires us to generalize
our categorical setting
to include CohFT of Kontsevich-Manin
\cite{KM}. 
Then we will be able to show that 
this unique functor actually generates
all \emph{Frobenius objects}
 of  $(\Vect,\tensor, K)$.
This topic will be treated in our forthcoming 
paper.
\end{rem}

Let us consider the monoidal (not full)
subcategory 
$\cA \subset (\Vect, \tensor, K)$ consisting of 
 commutative Frobenius algebras. 

\begin{thm}[Construction of 2D TQFTs]
There is a canonical Frobenius ECO
functor
\be\label{Omega-category}
\Omega:\cC\cG \lrar \cF un(\cA/K,  \cA/K).
\ee
When we start with a Frobenius algebra
$A$, this functor generates
a network of multilinear maps
 $$
 \Omega_A(\gam):A^{\tensor n}\lrar K
 $$
for all cell graphs $\gam\in \Gam_{g,n}$
for all values of $(g,n)$. This is the 
2D TQFT corresponding to 
the Frobenius algebra $A$. 
\end{thm}

\begin{proof}
This follows from the graph independence of 
Theorem~\ref{thm:independence}. 
\end{proof}


\begin{ack}
The authors  are grateful to  
the American Institute of Mathematics in California, 
the Banff International Research Station,
 the Institute for Mathematical
Sciences at the National University of Singapore,
 Kobe University, Leibniz Universit\"at 
 Hannover, the Lorentz Center for Mathematical 
 Sciences, Leiden, 
 Max-Planck-Institut f\"ur Mathematik in Bonn, 
 Mathematisches Forschungsinstitut Oberwolfach, 
 and Institut Henri Poincar\'e,
for their hospitality and financial support during
the authors' stay for collaboration related to the
subjects of this paper.
They   thank Ruian Chen,
Maxim Kontsevich and Shintaro Yanagida
for valuable discussions. They also thank
the referee for useful comments in improving
the manuscript.
O.D.\ thanks the Perimeter Institute for 
Theoretical Physics, and
M.M.\  thanks the Hong Kong University of 
Science and Technology and
the Simons Center for Geometry 
and Physics,
 for
financial support and 
hospitality. 
During the preparation of this work,
the research of O.D.\ has been supported by
 GRK 1463 \emph{Analysis,
Geometry, and String Theory} at 
Leibniz Universit\"at 
 Hannover, and Max-Planck-Institut
 f\"ur Mathematik, Bonn. 
The research of M.M.\ has been supported 
by 
NSF grants DMS-1104734, DMS-1309298, 
DMS-1619760, DMS-1642515,
and NSF-RNMS: Geometric Structures And 
Representation Varieties (GEAR Network, 
DMS-1107452, 1107263, 1107367).
\end{ack}


\providecommand{\bysame}{\leavevmode\hbox to3em{\hrulefill}\thinspace}

\bibliographystyle{amsplain}

\end{document}